\documentclass{amsart}
\usepackage{amsmath,amsfonts,latexsym,amssymb}
\usepackage[active]{srcltx} 
\newcommand{\bfs}{\boldsymbol}

\newtheorem{theorem}{Theorem}[section]

\newtheorem{corollary}[theorem]{Corollary}
\newtheorem{lemma}[theorem]{Lemma}
\newtheorem{proposition}[theorem]{Proposition}

\newcommand{\N}{\mathbb N}
\newcommand{\Z}{\mathbb Z}

\newcommand{\F}{\mathbb F}
\newcommand{\A}{\mathbb A}
\newcommand{\Pp}{\mathbb P}

\newcommand{\K}{{\sf K}}

\newcommand{\fq}{\F_{\hskip-0.7mm q}}
\newcommand{\cfq}{\overline{\F}_{\hskip-0.7mm q}}

\def\ifm#1#2{\relax \ifmmode#1\else#2\fi}

\newcommand{\klk}    {\ifm {,\ldots,} {$,\ldots,$}}

\newcommand{\plp}    {\ifm {+\cdots+} {$+\ldots+$}}

\newcommand{\om}[2]   {{#1}_0 \klk {#1}_{#2}}

\newcommand{\xon}    {\ifm {\om X n} {$\om X n$}}

\begin{document}
\title[Explicit estimates for singular complete intersections]{Explicit
estimates for the number of rational points of singular complete
intersections over a finite field}
\author[G. Matera et al.]{
%
Guillermo Matera${}^{1,2}$,
%
Mariana P\'erez${}^1$,
%
Melina Privitelli${}^2$,}

\address{${}^{1}$Instituto del Desarrollo Humano,
Universidad Nacional de Gene\-ral Sarmiento, J.M. Guti\'errez 1150
(B1613GSX) Los Polvorines, Buenos Aires, Argentina}
\email{\{gmatera,\,vperez\}@ungs.edu.ar}
\address{${}^{2}$ National Council of Science and Technology (CONICET),
Ar\-gentina}
\address{${}^{2}$Instituto de Ciencias,
Universidad Nacional de Gene\-ral Sarmiento, J.M. Guti\'errez 1150
(B1613GSX) Los Polvorines, Buenos Aires, Argentina}
\email{mprivitelli@conicet.gov.ar}

\thanks{The authors were partially supported by the grants
PIP CONICET 11220090100421 and UNGS 30/3180.}%


\date{\today}
\maketitle

\begin{abstract}
Let $V\subset\Pp^n(\cfq)$ be a complete intersection defined over a
finite field $\fq$ of dimension $r$ and singular locus of dimension
at most $0\le s\le r-2$. We obtain an explicit version of the
Hooley--Katz estimate $||V(\fq)|-p_r|=\mathcal{O}(q^{(r+s+1)/2})$,
where $|V(\fq)|$ denotes the number of $\fq$--rational points of $V$
and $p_r:=|\Pp^r(\fq)|$. Our estimate improves all the previous
estimates in several important cases. Our approach relies on tools
of classical algebraic geometry. A crucial ingredient is a new
effective version of the Bertini smoothness theorem, namely an
explicit upper bound of the degree of a proper Zariski closed subset
of $(\Pp^n)^{s+1}(\cfq)$ which contains all the singular linear
sections of $V$ of codimension $s+1$.
\end{abstract}
%
%
\section{Introduction}

Let $\fq$ be the finite field of $q$ elements and let $\cfq$ be the
algebraic closure of $\fq$. By 
$\Pp^n:=\Pp^n(\cfq)$ 
and $\A^n:=\A^n(\cfq)$ we denote the $n$--dimensional projective and
affine spaces defined over $\cfq$ respectively. For any affine or
projective variety $V$, we denote by $V(\fq)$ the set of
$\fq$--rational points of $V$, that is, the set of points of $V$
with coordinates in $\fq$, and by $|V(\fq)|$ its cardinality. In
particular, it is well--known that, for $r\ge 0$,
$$p_r:=|\Pp^r(\fq)|= q^r + \cdots + q + 1.$$

Let $V\subset\Pp^n$ be an ideal--theoretic complete intersection
defined over $\fq$, of dimension $r$ and multidegree
$\boldsymbol{d}:=(d_1\klk d_{n-r})$. In a fundamental work
\cite{Deligne74}, P. Deligne showed that, if $V$ is nonsingular,
then
\begin{equation}\label{eq: estimate deligne intro}
\big||V(\fq)|-p_r\big|\le b_r'(n,\boldsymbol{d})\,q^{\frac{r}{2}},
\end{equation}
where $b_r'(n,\boldsymbol{d})$ is the $r$th primitive Betti number
of any nonsingular complete intersection of $\Pp^n$ of dimension $r$
and multidegree $\boldsymbol{d}$ (see \cite[Theorem 4.1]{GhLa02a}
for an explicit expression of $b_r'(n,\boldsymbol{d})$ in terms of
$n$, $r$ and $\boldsymbol{d}$).

This result was extended by C. Hooley and N. Katz to singular
complete intersections. More precisely, in \cite{Hooley91} it is
proved that, if the singular locus of $V$ has dimension at most $s$,
then
\begin{equation}\label{eq: estimate hooley-katz}
|V(\fq)|=p_r+ \mathcal{O}(q^{\frac{r+s+1}{2}}),
\end{equation}
where the constant implied by the $\mathcal{O}$--notation depends
only on $n$, $r$ and $\boldsymbol{d}$, and it is not explicitly
given. Finally, S. Ghorpade and G. Lachaud obtained the following
explicit version of the Hooley--Katz bound \eqref{eq: estimate
hooley-katz} in \cite{GhLa02a}:
\begin{equation}\label{eq: estimate GL intro}
\big||V(\fq)|-p_r\big|\le
b_{r-s-1}'(n-s-1,\boldsymbol{d})\,q^{\frac{r+s+1}{2}}+
C(n,r,\boldsymbol{d})\,q^{\frac{r+s}{2}},
\end{equation}
where $C(n,r,\boldsymbol{d}):=9\cdot
2^{n-r}\big((n-r)d+3\big)^{n+1}$ and $d:=\max_{1\le i\le n-r}d_i$.

For the potential applications of \eqref{eq: estimate GL intro}, the
fact that the constant $C(n,r,\boldsymbol{d})$ depends exponentially
on the dimension $n$ of the ambient space $\Pp^n$ may be
inconvenient. This can be seen for example in \cite{CeMaPePr14} and
\cite{MaPePr14}, where we use estimates on the number of
$\fq$--rational points of singular complete intersection to
determine the behavior of the average value set of families of
univariate polynomials with prescribed coefficients over $\fq$. For
this reason, in this paper we obtain another explicit estimate on
$|V(\fq)|$ where this exponential dependency on $n$ is avoided.

From a methodological point of view, the estimates in \cite{GhLa02a}
are based on the Grothendieck--Lefschetz Trace Formula, together
with estimates for the dimension of certain spaces of \'etale
$\ell$--adic cohomology associated with the complete intersection
$V\subset\Pp^n$ under consideration.

Our approach is rather different and relies on tools of classical
projective algebraic geometry, combined with Deligne's estimate
\eqref{eq: estimate deligne intro}. The crucial geometric ingredient
is an effective version of the Bertini smoothness theorem, namely we
obtain quantitative information on the number of $\fq$--definable
linear sections $\mathcal{L}\subset\Pp^n$ such that
$V\cap\mathcal{L}$ has codimension $s+1$ and is nonsingular (see
Theorem \ref{th: definition H}).
\begin{theorem}\label{theorem: effective Bertini intro}
Let $V\subset \Pp^n$ be a complete intersection defined over $\fq$,
of dimension $r$, multidegree
$\boldsymbol{d}:=(d_1,\ldots,d_{n-r})$, degree $\delta$ and singular
locus of dimension at most $0\le s\le r-2$. Let
$D:=\sum_{i=1}^{n-r}(d_i-1)$. There exists a hypersurface
$\mathcal{H}\subset(\Pp^n)^{s+1}$, defined by a multihomogeneous
polynomial of degree at most $D^{r-s-1}(D+r-s)\delta$ in each group
of variables, with the following property: if
$\boldsymbol{\gamma}\in(\Pp^n)^{s+1}\setminus\mathcal{H}$ and
$\mathcal{L}:=\{\bfs\gamma\cdot x=0\}$, then $V\cap \mathcal{L}$ is
nonsingular of pure dimension $r-s-1$.
\end{theorem}

\cite{Ballico03} and \cite{CaMa07} provide effective versions of the
Bertini smoothness theorem for hypersurfaces and normal complete
intersections respectively. Theorem \ref{theorem: effective Bertini
intro} significantly improves and generalizes both results. We also
remark that a different variant of an effective Bertini smoothness
theorem is obtained in \cite{CaMaPr13}.

Combining Theorem \ref{theorem: effective Bertini intro} with upper
bounds on the number of $\fq$--rational zeros of multihomogeneous
polynomials we obtain rather precise estimates on the number of
nonsingular $\fq$--definable linear sections of codimension $s+1$ of
$V$. Then the analysis of the second moment of the number of
$\fq$--rational points of $V$ in linear sections of codimension
$s+1$ yields an estimate on the number of $\fq$--rational points of
$V$. More precisely, we obtain the following result (see Theorem
\ref{th: estimate hooley}).
\begin{theorem}\label{theorem: Hooley intro}
Assume that $q> 2(s+1)D^{r-s-1}(D+r-s)\delta$ and let $V \subset
\Pp^n$ a complete intersection defined over $\fq$, of dimension $r$,
multidegree $\boldsymbol{d}:=(d_1,\dots,d_{n-r})$, degree $\delta$
and singular locus of dimension at most $0\le s \le r-2$. Then
\begin{equation}\label{eq: estimate Hooley intro}
\big||V(\fq)|-p_r\big|\leq \big(b'_{r-s-1}(n-s-1,\boldsymbol{d}) +2
\sqrt{\delta}+1\big)\,q^{\frac{r+s+1}{2}}.
\end{equation}
\end{theorem}

According to \cite[Proposition 4.2]{GhLa02a}, the Betti number
$b'_{r-s-1}(n-s-1,\boldsymbol{d})$ can be bounded from above by a
quantity which is roughly of order $D^{r-s}\delta$. It follows that
the error term of \eqref{eq: estimate Hooley intro} depends linearly
on $\delta$. In this sense, \eqref{eq: estimate Hooley intro}
significantly improves \eqref{eq: estimate GL intro}, whose error
term may include an exponential term $\delta^{n+1}$ when $V$ is a
hypersurface. On the other hand, \eqref{eq: estimate GL intro} is
valid without restrictions on $q$, while \eqref{eq: estimate Hooley
intro} only holds for $q>2(s+1)D^{r-s-1}(D+r-s)\delta$.

The paper is organized as follows. In Section \ref{section: notions
and notations} we include a brief review of the notions of classical
algebraic geometry which we use. We also obtain an upper bound on
the number of $\fq$--rational zeros of a multihomogeneous
polynomial. Section \ref{section: Bertini} is devoted to the proof
of Theorem \ref{theorem: effective Bertini intro}. Finally, in
Section \ref{section: Hooley} we combine the upper bound of Section
\ref{section: notions and notations} with Theorem \ref{theorem:
effective Bertini intro} and the analysis of the second moment
mentioned before to prove Theorem \ref{theorem: Hooley intro}.
Taking into account that the condition on $q$ of the statement of
Theorem \ref{theorem: Hooley intro} may restrict its applicability,
we obtain a further estimate for normal complete intersections which
is valid without restrictions on $q$ (Corollary \ref{coro: estimate
normal variety}).
%
%
\section{Notions, notations and preliminary results}
\label{section: notions and notations} We use standard notions and
notations of commutative algebra and algebraic geometry as can be
found in, e.g., \cite{Harris92}, \cite{Kunz85} or
\cite{Shafarevich94}.

%
%
Let $\K$ be any of the fields $\fq$ or $\cfq$. We denote by $\A^n$
the affine $n$--dimensional space $\cfq{\!}^{n}$ and by $\Pp^n$ the
projective $n$--dimensional space over $\cfq{\!}^{n+1}$. Both spaces
are endowed with their respective Zariski topologies over $\K$, for
which a closed set is the zero locus of a set of polynomials of
$\K[X_1,\ldots, X_{n}]$, or of a set of homogeneous polynomials of
$\K[X_0,\ldots, X_{n}]$.

A subset $V\subset \Pp^n$ is a {\em projective variety defined over}
$\K$ (or a projective $\K$--variety for short) if it is the set of
common zeros in $\Pp^n$ of homogeneous polynomials $F_1,\ldots, F_m
\in\K[X_0 ,\ldots, X_n]$. Correspondingly, an {\em affine variety of
$\A^n$ defined over} $\K$ (or an affine $\K$--variety for short) is
the set of common zeros in $\A^n$ of polynomials $F_1,\ldots, F_{m}
\in \K[X_1,\ldots, X_{n}]$. We think a projective or affine
$\K$--variety to be equipped with the induced Zariski topology. We
shall frequently denote by $V(F_1\klk F_m)$ or $\{F_1=0\klk F_m=0\}$
the affine or projective $\K$--variety consisting of the common
zeros of the polynomials $F_1\klk F_m$.

In the remaining part of this section, unless otherwise stated, all
results referring to varieties in general should be understood as
valid for both projective and affine varieties.

A $\K$--variety $V$ is $\K$--{\em irreducible} if it cannot be
expressed as a finite union of proper $\K$--subvarieties of $V$.
Further, $V$ is {\em absolutely irreducible} if it is
$\cfq$--irreducible as a $\cfq$--variety. Any $\K$--variety $V$ can
be expressed as an irredundant union $V=\mathcal{C}_1\cup
\cdots\cup\mathcal{C}_s$ of irreducible (absolutely irreducible)
$\K$--varieties, unique up to reordering, which are called the {\em
irreducible} ({\em absolutely irreducible}) $\K$--{\em components}
of $V$.

For a $\K$-variety $V$ contained in $\Pp^n$ or $\A^n$, we denote by
$I(V)$ its {\em defining ideal}, namely the set of polynomials of
$\K[X_0,\ldots, X_n]$, or of $\K[X_1,\ldots, X_n]$, vanishing on
$V$. The {\em coordinate ring} $\K[V]$ of $V$ is defined as the
quotient ring $\K[X_0,\ldots,X_n]/I(V)$ or
$\K[X_1,\ldots,X_n]/I(V)$. The {\em dimension} $\dim V$ of $V$ is
the length $r$ of the longest chain $V_0\varsubsetneq V_1
\varsubsetneq\cdots \varsubsetneq V_r$ of nonempty irreducible
$\K$-varieties contained in $V$. We call $V$ {\em equidimensional}
if all its irreducible $\K$--components are of the same dimension.
We say that $V$ has {\em pure dimension} $r$ if it is
equidimensional of dimension $r$.

A $\K$--variety of $\Pp^n$ or $\A^n$ of pure dimension $n-1$ is
called a $\K$--{\em hypersurface}. A $\K$--hypersurface of $\Pp^n$
(or $\A^n$) is the set of zeros of a single nonzero polynomial of
$\K[X_0\klk X_n]$ (or of $\K[X_1\klk X_n]$).

The {\em degree} $\deg V$ of an irreducible $\K$-variety $V$ is the
maximum number of points lying in the intersection of $V$ with a
linear space $L$ of codimension $\dim V$, for which $V\cap L$ is a
finite set. More generally, following \cite{Heintz83},
if $V=\mathcal{C}_1\cup\cdots\cup \mathcal{C}_s$ is the
decomposition of $V$ into irreducible $\K$--components, we define
its degree as
$$\deg V:=\sum_{i=1}^s\deg \mathcal{C}_i.$$
The degree of a $\K$--hypersurface $V$ is the degree of a polynomial
of minimal degree defining $V$. 
%

%
%
%

Let $V\subset\A^n$ be a $\K$--variety and let $I(V)\subset
\K[X_1,\ldots, X_n]$ be the defining ideal of $V$. Let $x$ be a
point of $V$. The {\em dimension} $\dim_xV$ {\em of} $V$ {\em at}
$x$ is the maximum of the dimensions of the irreducible
$\K$--components of $V$ that contain $x$. If $I(V)=(F_1,\ldots,
F_m)$, the {\em tangent space} $\mathcal{T}_xV$ to $V$ at $x$ is the
kernel of the Jacobian matrix $(\partial F_i/\partial X_j)_{1\le
i\le m,1\le j\le n}(x)$ of the polynomials $F_1,\ldots, F_m$ with
respect to $X_1,\ldots, X_n$ at $x$. We have the following
inequality (see, e.g., \cite[page 94]{Shafarevich94}):
$$\dim\mathcal{T}_xV\ge \dim_xV.$$
The point $x$ is {\em regular} if $\dim\mathcal{T}_xV=\dim_xV$.
Otherwise, the point $x$ is called {\em singular}. The set of
singular points of $V$ is the {\em singular locus}
$\mathrm{Sing}(V)$ of $V$; it is a closed $\K$--subvariety of $V$. A
variety is called {\em nonsingular} if its singular locus is empty.
For a projective variety, the concepts of tangent space, regular and
singular point can be defined by considering an affine neighborhood
of the point under consideration.
%
%
\subsection{Complete intersections}
A $\K$--variety $V$ of dimension $r$ in the $n$--dimensional space
is an (ideal--theoretic) {\em complete intersection} if its ideal
$I(V)$ over $\K$ can be generated by $n-r$ polynomials. If
$V\subset\Pp^n$ is a complete intersection defined over $\K$, of
dimension $r$ and degree $\delta$, and $F_1 \klk F_{n-r}$ is a
system of homogeneous generators of $I(V)$, the degrees $d_1\klk
d_{n-r}$ depend only on $V$ and not on the system of homogeneous
generators. Arranging the $d_i$ in such a way that $d_1\geq d_2 \geq
\cdots \geq d_{n-r}$, we call $\boldsymbol{d}:=(d_1\klk d_{n-r})$
the {\em multidegree} of $V$.

If $V\subset\Pp^n$ is a complete intersection of
multidegree $\boldsymbol{d}:=(d_1\klk d_{n-r})$, then 
the {\em B\'ezout theorem} (see, e.g., \cite[Theorem 18.3]{Harris92}
or \cite[\S 5.5, page 80]{SmKaKeTr00}) asserts that
$$\deg V=d_1\cdots d_{n-r}.$$

We shall consider a particular class of complete intersections,
which we now define. A $\K$--variety is {\em regular in codimension
$m$} if the singular locus $\mathrm{Sing}(V)$ of $V$ has codimension
at least $m+1$ in $V$, namely if $\dim V-\dim \mathrm{Sing}(V)\ge
m+1$. A complete intersection $V$ which is regular in codimension 1
is called {\em normal} (actually, normality is a general notion that
agrees on complete intersections with the one we define here). A
fundamental result for projective complete intersections is the
Hartshorne connectedness theorem (see, e.g., \cite[Theorem
VI.4.2]{Kunz85}), which we now state. If $V\subset\Pp^n$ is a
complete intersection defined over $\K$ and $W\subset V$ is any
$\K$--subvariety of codimension at least 2, then $V\setminus W$ is
connected in the Zariski topology of $\Pp^n$ over $\K$. Applying the
Hartshorne connectedness theorem with $W:=\mathrm{Sing}(V)$, one
deduces the following result.
\begin{theorem}\label{th: normal complete int implies irred}
If $V\subset\Pp^n$ is a normal complete intersection, then $V$ is
absolutely irreducible.
\end{theorem}
%
%
\subsection{Rational points}
Let $\Pp^n(\fq)$ be the $n$--dimensional projective space over $\fq$
and let $\A^n(\fq)$ be the $n$--dimensional $\fq$--vector space
$\fq^n$. For a projective variety $V\subset\Pp^n$ or an affine
variety $V\subset\A^n$, we denote by $V(\fq)$ the set of
$\fq$--rational points of $V$, namely $V(\fq):=V\cap \Pp^n(\fq)$ or
$V(\fq):=V\cap \A^n(\fq)$ respectively.

For a projective variety $V$ of dimension $r$ and degree $\delta$,
we have (see \cite[Proposition 12.1]{GhLa02a} or \cite[Proposition
3.1]{CaMa07}):
 \begin{equation}\label{eq: upper bound -- projective gral}
   |V(\fq)|\leq \delta p_r.
 \end{equation}
On the other hand, if $V$ is an affine variety of dimension $r$ and
degree $\delta$, then (see, e.g., \cite[Lemma 2.1]{CaMa06})
 \begin{equation}\label{eq: upper bound -- affine gral}
   |V(\fq)|\leq \delta q^r.
 \end{equation}
%
%
\subsection{Multiprojective space}
\label{subsec: multiprojective space}
Let $\N:=\Z_{\ge 0}$ be the set of nonnegative integers. For
$\boldsymbol{n}:=(n_1\klk n_m)\in\N^m$, we define
$|\boldsymbol{n}|:=n_1\plp n_m$. 
%
Denote by $\Pp^{\boldsymbol{n}} :=\Pp^{\boldsymbol{n}}(\cfq)$ the
multiprojective space $\Pp^{\boldsymbol{n}}:=
\Pp^{n_1}\times\cdots\times \Pp^{n_m}$. For $1\le i\le m$, let
$\Gamma_i:=\{\Gamma_{i,0}\klk \Gamma_{i,n_i}\}$ be group of $n_i+1$
variables and let $\bfs \Gamma:=\{\Gamma_1\klk \Gamma_m\}$. For
$\K:=\cfq$ or $\K:=\fq$, a {\em multihomogeneous} polynomial of
$\K[\bfs \Gamma]$ of multidegree $\boldsymbol{d}:=(d_1\klk d_m)$ is
an element which is homogeneous of degree $d_i$ in $\Gamma_i$ for
$1\le i\le m$. An ideal $I\subset\K[\bfs \Gamma]$ is {\em
multihomogeneous} if it is generated by a family of multihomogeneous
polynomials. For any such ideal, we denote by
$V(I)\subset\Pp^{\boldsymbol{n}}$ the variety defined as its set of
common zeros. In particular, a hypersurface in
$\Pp^{\boldsymbol{n}}$ defined over $\K$ is the set of zeros of a
multihomogeneous polynomial of $\K[\bfs \Gamma]$. The notions of
irreducible variety and dimension of a subvariety of
$\Pp^{\boldsymbol{n}}$ are defined as in the projective space.
\subsubsection{Number of zeros of multihomogeneous hypersurfaces}
\label{subsec: zeros multihomogeneous pols}
With notations as above, let $\fq^{\boldsymbol{n}+\bfs
1}:=\fq^{n_1+1}\times\cdots \times \fq^{n_m+1}$. 
Let $F\in\cfq[\bfs\Gamma]$ be a multihomogeneous polynomial of
multidegree $\boldsymbol{d}:=(d_1\klk d_m)$. In this section we
establish two basic results concerning the number of $\fq$--rational
zeros of $F$. The first result is a nontrivial upper bound on the
number of zeros of $F$ in $\fq^{\boldsymbol{n}+\bfs 1}$, which
improves (\ref{eq: upper bound -- affine gral}) for multiprojective
hypersurfaces.

For $\boldsymbol{\alpha}\in\N^m$, we denote
$\boldsymbol{d}^{\boldsymbol{\alpha}}:=d_1^{\,\alpha_1}\cdots
d_m^{\,\alpha_m}$. Further, let
$$\eta_m(\boldsymbol{d},\boldsymbol{n})
:=\sum_{\boldsymbol{\varepsilon}\in\{0,1\}^m\setminus\{\boldsymbol{0}\}}
(-1)^{|\boldsymbol{\varepsilon}|+1}\boldsymbol{d}^{\,
\boldsymbol{\varepsilon}}
q^{|\boldsymbol{n}|+m-|\bfs\varepsilon|}.$$
Observe that $\eta_m(\boldsymbol{d},\boldsymbol{n})< q^{|\bfs
n|+m}=|\fq^{\boldsymbol{n}+\bfs 1}|$ if $q> \max_{1\le i\le m}d_i$,
while this inequality may not hold for $q\le \max_{1\le i\le m}d_i$.
We have the following result.
\begin{proposition}\label{prop: upper bound multihomogeneous}
Let $F\in\cfq[\bfs \Gamma]$ be a multihomogeneous polynomial of
multidegree $\boldsymbol{d}$ with $\max_{1\le i\le m}d_i<q$ and let
$N$ be the number of zeros of $F$ in $\fq^{\boldsymbol{n}+\bfs 1}$.
Then
$$N\le \eta_m(\boldsymbol{d},\boldsymbol{n}).$$
\end{proposition}
\begin{proof}
We argue by induction on $m$. The case $m=1$ is (\ref{eq: upper
bound -- affine gral}).

Suppose that the statement holds for $m-1$ and let $F\in\cfq[\bfs
\Gamma]$ be an $m$--homogeneous polynomial of multidegree
$\boldsymbol{d}:=(d_1\klk d_m)$. Let $N$ be the number of zeros of
$F$ in $\fq^{\boldsymbol{n}+\bfs 1}$, and let $Z_m$ be the set of
$\gamma_m$ in $\fq^{n_m+1}$ such that the substitution
$F(\Gamma_1\klk \Gamma_{m-1},\gamma_m)$ of $\gamma_m$ for $\Gamma_m$
in $F$ yields the zero polynomial of
$\cfq[\Gamma_1\klk\Gamma_{m-1}]$. Consider $F$ as an element of
$\cfq[\Gamma_m][\Gamma_1\klk\Gamma_{m-1}]$ and let
$A\in\cfq[\Gamma_m]$ be a nonzero homogeneous polynomial of degree
$d_m$ which occurs as the coefficient of a monomial
$\Gamma_1^{\boldsymbol{\alpha}_1}\cdots
\Gamma_{m-1}^{\boldsymbol{\alpha}_{m-1}}$ in the dense
representation of $F$. As $Z_m$ is contained in the set of zeros in
$\fq^{n_m+1}$ of $A$, by (\ref{eq: upper bound -- affine gral}) we
have $|Z_m|\le d_mq^{n_m}$.

Since $d_m<q$ by hypothesis, it follows that $|Z_m|\le
d_mq^{n_m}<q^{n_m+1}=|\fq^{n_m+1}|$, which implies that
$\fq^{n_m+1}\setminus Z_m$ is nonempty. Fix
$\gamma_m\in\fq^{n_m+1}\setminus Z_m$ and denote by $N_{m-1}$ the
number of zeros of $F(\Gamma_1\klk \Gamma_{m-1},\gamma_m)$ in
$\fq^{n_1+1}\times\cdots\times \fq^{n_{m-1}+1}$. By the inductive
hypothesis and the fact that $\max_{1\le i\le m-1}d_i< q$, we see
that
$$N_{m-1}\le \eta_{m-1}(\boldsymbol{d}^*,\boldsymbol{n}^*)<
q^{|\bfs n^*|+m-1},$$
where $\boldsymbol{d}^*:=(d_1\klk d_{m-1})$ and
$\boldsymbol{n}^*:=(n_1\klk n_{m-1})$. As a consequence,
\begin{align*}
N &\le |Z_m| q^{|\bfs n^*|+m-1}+(q^{n_m+1}-|Z_m|)
\eta_{m-1}(\boldsymbol{d}^*,\boldsymbol{n}^*)\\
&= |Z_m|\left(q^{|\bfs n^*|+m-1}
-\eta_{m-1}(\boldsymbol{d}^*,\boldsymbol{n}^*)\right)+
\eta_{m-1}(\boldsymbol{d}^*,\boldsymbol{n}^*) q^{n_m+1}\le
\eta_m(\boldsymbol{d},\boldsymbol{n}).
\end{align*}
This completes the proof of the proposition.
\end{proof}

The second result is concerned with conditions of existence of a
point of $\Pp^{\boldsymbol{n}}(\fq)$ which does not annihilates $F$
and will be used to obtain an effective version of the Bertini
smoothness theorem 
(Theorem \ref{th: definition H}).

\begin{corollary}\label{coro: existencia no-cero pol multih}
Let $F\in\cfq[\bfs \Gamma]$ be a multihomogeneous polynomial of
multidegree $\bfs d$ and let $d:=\max_{1\le i\le m}d_i$. If $q>d$,
then there exists $\bfs \gamma\in \Pp^{\boldsymbol{n}}(\fq)$ with
$F(\bfs \gamma)\not=0$.
\end{corollary}
\begin{proof}
It suffices to show that there exists $\bfs \gamma'\in\fq^{\bfs
n+\bfs 1}$ with $F(\bfs \gamma')\not=0$. Let $N$ be the number of
zeros of $F$ in $\fq^{\boldsymbol{n}+\bfs 1}$. According to
Proposition \ref{prop: upper bound multihomogeneous}, the number
$N_{\ne 0}$ of elements in $\fq^{\boldsymbol{n}+\bfs 1}$ not
annihilating $F$ is bounded as follows:
$$N_{\ne 0}\ge q^{|\bfs n|+m}-
\eta_m(\boldsymbol{d},\boldsymbol{n}) =
\sum_{\boldsymbol{\varepsilon}\in\{0,1\}^m}
(-1)^{|\boldsymbol{\varepsilon}|}\boldsymbol{d}^{\,\boldsymbol{\varepsilon}}
q^{|\bfs n|+m-\boldsymbol{\varepsilon}}=
\prod_{i=1}^m(q^{n_i+1}-d_iq^{n_i}).$$
Since $q>d$, we have $q^{n_i+1}>d_iq^{n_i-1}$ for $1\le i\le m$,
which yields the corollary.
\end{proof}
%
%
\section{On the existence of nonsingular linear sections}
\label{section: Bertini}
In this section we establish a Bertini--type theorem, namely we show
the existence of nonsingular linear sections of a singular complete
intersection. The Bertini smoothness theorem asserts that a generic
hyperplane section of a nonsingular variety $V$ is nonsingular. A
more precise variant of this result asserts that, if $V\subset\Pp^n$
is a projective variety with singular locus of dimension at most
$s$, then the section of $V$ defined by a generic linear space of
$\Pp^n$ of codimension at least $s+1$ is nonsingular (see, e.g.,
\cite[Proposition 1.3]{GhLa02a}). Identifying each section of this
type with a point in the multiprojective space $(\Pp^n)^{s+1}$, we
show the existence of a hypersurface
$\mathcal{H}\subset(\Pp^n)^{s+1}$ containing all the linear
subvarieties of codimension $s+1$ of $(\Pp^n)^{s+1}$ which yield
singular sections of $V$. We also estimate the multidegree of this
hypersurface.

Let $V \subset \Pp^n$ be a complete intersection defined by
homogeneous polynomials $F_1\klk F_{n-r} \in \fq[\xon]$ of degrees
$d_1\ge\cdots\ge d_{n-r}\ge 2$ respectively.  Let
$\Sigma:=\mathrm{Sing}\,V$ and suppose that there exists $s$ with
$0\le s\le r-2$ such that $\dim\Sigma\le s$. In particular, $V$ is a
normal complete intersection, and therefore absolutely irreducible
(Theorem \ref{th: normal complete int implies irred}). We denote by
$V_{\rm sm}:=V\setminus\Sigma$ the smooth locus of $V$. Finally, set
$\delta:=\deg V=d_1\cdots d_{n-r}$ and $D:=\sum_{i=1}^{n-r}(d_i-1)$.

Set $X:=(X_0\klk X_n)$. For $\mu:=(\mu_0:\dots:\mu_n)\in\Pp^n$, we
shall use the notation $\mu\cdot X:=\mu_0X_0\plp\mu_nX_n$. Let
$\boldsymbol{\gamma}:=(\gamma_0, \dots, \gamma_s)\in(\Pp^n)^{s+1}$,
where $\gamma_0\klk\gamma_s$ are $\cfq$--linearly independent, and
consider the linear variety $\mathcal{L} \subset \Pp^n$ defined by
$$\mathcal{L}:=\{\bfs\gamma\cdot x=0\}:=\{x\in\Pp^n:\gamma_0\cdot
x=\cdots=\gamma_s\cdot x=0\}.$$
Our goal is to prove the existence of a hypersurface $\mathcal{H}
\subset (\Pp^n)^{s+1}$ with the following property: if $\bfs
\gamma\in (\Pp^n)^{s+1} \setminus \mathcal{H}$ and
$\mathcal{L}:=\{\bfs \gamma\cdot x=0\}$, then $V \cap \mathcal{L}$
is nonsingular of pure dimension $r-s-1$.

Let $\Gamma_i:=(\Gamma_{i,0}\klk \Gamma_{i,n})$ be a group of $n+1$
variables for $0\le i\le s$ and denote
$\boldsymbol{\Gamma}:=(\Gamma_0\klk\Gamma_s)$. We consider the
incidence variety
\begin{align*}
\mathcal{W}& :=(V_{\rm sm} \times \mathcal{U})\cap\{\Gamma_0\cdot
X=0,\dots,\Gamma_{s}\cdot X=0, \Delta_1(\boldsymbol{\Gamma},X)
=0,\dots,\Delta_m(\boldsymbol{\Gamma},X) =0\},
\end{align*}
where $\mathcal{U}\subset(\Pp^n)^{s+1}$ is the Zariski open subset
of  $(s+ 1)\times (n+1)$-matrices of maximal rank and
$\Delta_1\klk\Delta_m$ are the maximal minors of the matrix
\begin{equation}\label{eq: matrix with gammas}
\mathcal{M}(X,\boldsymbol{\Gamma}):=
     \begin{pmatrix}
            \frac{\partial F_1}{\partial X_0} &
            \dots & \frac{\partial F_1}{\partial X_n}
            \cr \vdots & & \vdots \cr
            \frac{\partial F_{n-r}}{\partial X_0} &
            \dots & \frac{\partial F_{n-r}}{\partial X_n} \cr
            \Gamma_{0,0} & \dots & \Gamma_{0,n} \cr
            \vdots & & \vdots \cr
            \Gamma_{s,0} & \dots & \Gamma_{s,n}
     \end{pmatrix}.
\end{equation}
Let $l:=n(s+1)$. The following result summarizes the main properties
of the incidence variety $\mathcal{W}$ we shall use.
\begin{proposition}\label{prop: W es irreducible}
$\mathcal{W}$ is an absolutely irreducible subvariety of
$V_{\mathrm{sm}}\times \mathcal{U}$ of dimension $l-1$.
\end{proposition}
\begin{proof}
Let $\pi_1:\mathcal{W}\to V_{\rm sm}$ be the mapping
$\pi_1(x,\boldsymbol{\gamma}):= x$. Fix $x\in V_{\rm sm}$ and
consider the fiber $\pi_1^{-1}(x)$. We have
$\pi_1^{-1}(x)=\{x\}\times \Omega$, where $\Omega \subset
\mathcal{U}$ is the set of
$\boldsymbol{\gamma}:=(\gamma_0,\dots,\gamma_{s})$ such that
$\gamma_0\cdot x=\cdots=\gamma_{s}\cdot x=0$ and the matrix $
\mathcal{M}(x,\bfs \gamma)$ is not of full rank. The latter
condition is equivalent to
\begin{equation}\label{eq: cond var incidencia singularidad multih}
\langle\gamma_0,\dots,\gamma_{s}\rangle \cap \big\langle\nabla
F_1(x)\klk\nabla F_{n-r}(x)\big\rangle \neq \{\boldsymbol{0}\} ,
\end{equation}
where $\langle v_0\klk v_m\rangle \subset \A^{n+1}$ denotes the
linear variety spanned by $v_0\klk v_m$ in $\A^{n+1}$. Let
$\mathbb{V}:=\{v\in\A^{n+1}:v\cdot x=0\}$. Observe that $\nabla
F_j(x)\in\mathbb{V}$ for $1\le j\le n-r$. Then (\ref{eq: cond var
incidencia singularidad multih}) holds if and only if
$\gamma_0\klk\gamma_{s}$ are not linearly independent in the
quotient $\cfq$--vector space
$$
\mathbb{W}:=\mathbb{V}/\langle\nabla F_1(x)\klk\nabla
F_{n-r}(x)\rangle.
$$
This shows that the affine cone of $\Omega$ is, modulo
$\langle\nabla F_1(x)\klk \nabla F_{n-r}(x)\rangle$, isomorphic to
the Zariski open set $L_{s}'(\A^{s+1},\mathbb{W})\cap
\Phi(\mathcal{U}_{\rm aff})$ of $L_{s}'(\A^{s+1},\mathbb{W})$, where
$$
L_{s}'(\A^{s+1},\mathbb{W}):=\{f\in{\rm
Hom}_{\cfq}(\A^{s+1},\mathbb{W}):{\rm rank}(f)\le s\},
$$
$\mathcal{U}_{\rm aff}\subset (\A^{n+1})^{s+1}$ is the multi--affine
cone of $\mathcal{U}$ and $\Phi:{\rm
Hom}_{\cfq}(\A^{s+1},\A^{n+1})\to {\rm
Hom}_{\cfq}(\A^{s+1},\mathbb{W})$ is the surjective map induced by
the quotient map $\A^{n+1}\to\mathbb{W}$.

According to \cite[Proposition 1.1]{BrVe88},
$L_{s}'(\A^{s+1},\mathbb{W})$ is an absolutely irreducible variety
of dimension $s(r+1)$.  Since we are considering subspaces of
$\mathbb{V}$ of dimension $s+1$ modulo $\langle\nabla
F_1(x)\klk\nabla F_{n-r}(x)\rangle$, which has dimension $n-r$
because $x\in V_{\rm sm}$, it follows that the affine cone of
$\pi_1^{-1}(x)=\{x\}\times \Omega$ is an open dense subset of an
irreducible variety of $V_{\rm sm}\times \mathcal{U}_{\rm aff}$ of
dimension $s(r+1)+(n-r)(s+1)=l+s-r$. This implies that
$\pi_1^{-1}(x)=\{x\}\times \Omega$ is an irreducible subvariety of
$V_{\rm sm}\times\mathcal{U}$ of dimension $l+s-r-(s+1)=l-r-1$.

We claim that the projection $V_{\rm sm}\times \mathcal{U}\to V_{\rm
sm}$ is closed. Indeed, this is the case if $\mathcal{U}$ is a
complete variety (see, e.g., \cite[Chapter 2, \S 3]{Danilov94}). It
is well--known that a projective variety is complete (see, e.g.,
\cite[Chapter 2, \S 3.3]{Danilov94}). Furthermore, if there exists a
proper map from a quasiprojective variety to a complete variety,
then the former is complete (see, e.g., \cite[Chapter 2, \S
3.2]{Danilov94}). In our case, it is not hard to see that the
mapping $\mathcal{U}\to\mathbb{G}(s,n)$ defined by the Pl\"ucker
coordinates is proper. Since $\mathbb{G}(s,n)$ is a projective
variety and thus complete, the claim follows.

Let $\mathcal{W}=\bigcup_{j}\mathcal{C}_j$ be the decomposition of
$\mathcal{W}$ into irreducible components. Our previous arguments
show that $\pi_1: \mathcal{W}\to V_{\rm sm}$ is surjective. Then
$\pi_1(\mathcal{W})=V_{\rm sm}=\bigcup_{j}\pi_1(\mathcal{C}_j)$ and
each $\pi_1(\mathcal{C}_j)$ is a closed subset of $V_{\rm sm}$.
Since $V$ is a normal complete intersection, and thus absolutely
irreducible (Theorem \ref{th: normal complete int implies irred}),
it follows that $V_{\rm sm}$ is absolutely irreducible and there
exists $j$ with $V_{\rm sm}=\pi_1(\mathcal{C}_j)$.

Now the proof follows {\em mutatis mutandis} the second and third
paragraph of the proof of \cite[\S I.6.3, Theorem 8]{Shafarevich94}
to conclude that $\mathcal{W}$ is an absolutely irreducible
subvariety of $V_{\rm sm}\times\mathcal{U}$.

Finally, by the Theorem on the dimension of fibers (see, e.g.,
\cite[\S I.6.3, Theorem 7]{Shafarevich94}), for any $x\in V_{\rm
sm}$ we have
$$l-r-1=\dim \pi_1^{-1}(x)=\dim \mathcal{W}-\dim V_{\rm sm}=\dim \mathcal{W}-r.$$
This shows that $\mathcal{W}$ has dimension $l-1$ and finishes the
proof.
\end{proof}

An immediate consequence of Proposition \ref{prop: W es irreducible}
is that the Zariski closure of the image of the projection
$\pi_2:\mathcal{W}\to\mathcal{U}$ on the second argument is an
absolutely irreducible variety of dimension at most $l-1$. Our
interest in the set $\pi_2(\mathcal{W})$ is based on the following
lemma.
\begin{lemma}\label{lemma: props section notin pi_2(W)}
If $\bfs\gamma\in\mathcal{U}\setminus\pi_2(\mathcal{W})$, then the
linear section $V_{\rm sm}\cap \mathcal{L}$ defined by
$\mathcal{L}:=\{\bfs\gamma\cdot x=0\}$ is nonsingular of pure
dimension $r-s-1$.
\end{lemma}
\begin{proof}
Fix $\bfs\gamma\in\mathcal{U}\setminus \pi_2(\mathcal{W})$ and
denote $\mathcal{L}:=\{\bfs\gamma\cdot x=0\}$. According to
\cite[Lemma 1.1]{GhLa02a}, $\mathrm{Sing}(V_{\rm sm}\cap
\mathcal{L})=N(V_{\rm sm},\mathcal{L})$, where $N(V_{\rm
sm},\mathcal{L})$ is the set of points $x \in V_{\rm sm}$ where $V$
and $\mathcal{L}$ do not meet transversely, that is,
$\mathrm{dim}\mathcal{T}_xV \cap \mathcal{L}> \mathrm{dim}
\mathcal{T}_x V - \mathrm{codim} \mathcal{L} = r-s-1$. For $x\in
V_{\rm sm}\cap\mathcal{L}$, we have $(x,\boldsymbol{\gamma})\notin
\mathcal{W}$ and then $\mathcal{M}(x,\boldsymbol{\gamma})$ has
maximal rank, where $\mathcal{M}(X,\boldsymbol{\Gamma})$ is the
matrix of (\ref{eq: matrix with gammas}). As a consequence,
$\mathrm{dim}\mathcal{T}_xV \cap \mathcal{L}= r-s-1$. This implies
that $V$ and $\mathcal{L}$ meet transversely at $x$, and hence $x$
is a nonsingular point of $V_{\rm sm}\cap \mathcal{L}$. This shows
that $V_{\rm sm}\cap \mathcal{L}$ is nonsingular.

By \cite[\S I.6.2, Corollary 5]{Shafarevich94}, each irreducible
component of $V_{\rm sm}\cap \mathcal{L}$ has dimension at least
$r-s-1$. For $x\in V_{\rm sm}\cap \mathcal{L}$, the matrix
$\mathcal{M}(x,\boldsymbol{\gamma})$ has maximal rank. Hence, $\dim
\mathcal{T}_x( V\cap \mathcal{L})\leq r-s-1$, which implies that
each irreducible component of $V_{\rm sm}\cap\mathcal{L}$ containing
$x$ has dimension at most $r-s-1$. We conclude that $V_{\rm
sm}\cap\mathcal{L}$ is of pure dimension $r-s-1$.
\end{proof}

We shall show that the set $\pi_2(\mathcal{W})$ is contained in a
hypersurface of $(\Pp^n)^{s+1}$ of ``low'' degree. Denote
$L_\Gamma:=\{\Gamma_0\cdot X=0,\dots,\Gamma_{s}\cdot
X=0\}\subset(\Pp^n)^{s+2}$ and let
$\mathcal{W}''\subset(\Pp^n)^{s+2}$ be the following variety:
\begin{equation}\label{eq: definition W''}
\mathcal{W}'':=\mathcal{W}\cup\big((\Sigma\times(\Pp^n)^{s+1})\cap
L_\Gamma\big)\cup\big((V\times ((\Pp^n)^{s+1}\setminus
\mathcal{U}))\cap L_\Gamma\big).
\end{equation}
We have the following result.
\begin{lemma}\label{lemma: props W''}
The variety $\mathcal{W}''$ has dimension $l-1$ and the following
identity holds:
\begin{equation}\label{eq: W'' - identity}
\mathcal{W}''=\big(\big(V\times (\Pp^n)^{s+1}\big)\cap
L_\Gamma\big)\cap\{ \Delta_1(\boldsymbol{\Gamma},X)=0,\dots,
\Delta_m(\boldsymbol{\Gamma},X)=0\}.
\end{equation}
\end{lemma}
\begin{proof}
First we prove \eqref{eq: W'' - identity}. It is easy to see that
the left--hand side is contained in the right--hand side. On the
other hand, for $(x,\boldsymbol{\gamma})\in \big(V\times
(\Pp^n)^{s+1}\big)\cap L_{\Gamma}$, either  $x\in \Sigma$, or
$\boldsymbol{\gamma}\in (\Pp^n)^{s+1}\setminus \mathcal{U}$, or
$(x,\gamma)\in V_{\mathrm{sm}}\times \mathcal{U}$. In the first two
cases, $(x,\boldsymbol{\gamma})\in\mathcal{W}''$ and the identity
$\Delta_j(x,\boldsymbol{\gamma})=0$ is satisfied for $1\leq j \leq
m$. In the third case we have $(x,\boldsymbol{\gamma})\in
\mathcal{W}''$ if and only if $\Delta_j(x,\boldsymbol{\gamma})=0$
for $1\leq j \leq m$. This shows the claim.

Next we determine the dimension of $\mathcal{W}''$. Observe that
$\Sigma\times(\Pp^n)^{s+1}$ is a cylinder whose intersection with
the equations $\Gamma_0\cdot X=0,\dots,\Gamma_{s}\cdot X=0$ has
codimension $s+1$. Hence, $(\Sigma\times(\Pp^n)^{s+1})\cap L_\Gamma$
has dimension at most $s+l-(s+1)=l-1$. On the other hand, the affine
cone of $(\Pp^n)^{s+1}\setminus\mathcal{U}$ is the closed set
$L_s(\A^{s+1},\A^{n+1})$ of matrices of rank at most $s$. By
\cite[Proposition 1.1]{BrVe88}, $\dim L_s(\A^{s+1},\A^{n+1})
=s(n+2)$; thus, $(\Pp^n)^{s+1}\setminus\mathcal{U}$ has dimension
$s(n+2)-(s+1)=l+s-n-1$. Then $V\times \big((\Pp^n)^{s+1}
\setminus\mathcal{U}\big)$ has dimension $r+l+s-n-1$. Consider the
projection $\pi_2:\big(V\times((\Pp^n)^{s+1}\setminus
\mathcal{U})\big)\cap L_\Gamma\to (\Pp^n)^{s+1}\setminus
\mathcal{U}$ on the second argument. The intersection of $V$ with a
generic linear variety of $\Pp^n$ of codimension $s$ is of pure
dimension $r-s$. Therefore, a generic fiber
$\pi_2^{-1}(\boldsymbol{\gamma})$ has dimension $r-s$ and the
theorem on the dimension of fibers shows that
$$r-s=\dim\pi_2^{-1}(\boldsymbol{\gamma})\ge
\dim\big(V\times((\Pp^n)^{s+1}\setminus \mathcal{U})\big)\cap
L_\Gamma-(l+s-n-1).$$
We deduce that $\dim \big(V\times((\Pp^n)^{s+1}\setminus
\mathcal{U})\big)\cap L_\Gamma\le l-n+r-1<l-1$. Combining these
facts with Proposition \ref{prop: W es irreducible} we conclude that
$\mathcal{W}''$ has dimension $l-1$.
\end{proof}

As an immediate consequence of Lemma \ref{lemma: props W''}, we
obtain the following result.
\begin{corollary}\label{coro: definition W'}
There exist linear combinations $\Delta^1\klk\Delta^{r-s}$ of the
maximal minors $\Delta_1(\boldsymbol{\Gamma},X) \klk
\Delta_m(\boldsymbol{\Gamma},X)$ of the matrix
$\mathcal{M}(X,\boldsymbol{\Gamma})$ of (\ref{eq: matrix with
gammas}) such that the variety $\mathcal{W}'\subset(\Pp^n)^{s+2}$
defined as the set of common solutions of
\begin{equation}\label{eq: definition W_s'}
F_1=0,\dots,F_{n-r}=0,\Gamma_0\cdot X=0,\dots,\Gamma_{s}\cdot X=0,
\Delta^1=0,\dots, \Delta^{r-s}=0,
\end{equation}
is of pure dimension $l-1$ and contains $\mathcal{W}''$.
\end{corollary}
\begin{proof}
The intersection of $V\times (\Pp^n)^{s+1}$ with $L_\Gamma$ has
codimension $s+1$, that is,
$$\dim \big(V\times (\Pp^n)^{s+1}\big)\cap
L_\Gamma=r+(n-1)(s+1)=r-s+l-1.$$
Then the result is an easy consequence of the fact that
$\mathcal{W}''$ has codimension at least $r-s$ in $\big(V\times
(\Pp^n)^{s+1}\big)\cap L_\Gamma$. Indeed, applying, e.g.,
{\cite[Lemma 4.4]{CaMaPr13}} to $\big(V\times (\Pp^n)^{s+1}\big)\cap
L_\Gamma$ and $\mathcal{W}''$, we readily deduce the corollary.
\end{proof}

Now we are in a position to prove that the main result of this
section, namely that the set of $\bfs\gamma\in (\Pp^n)^{s+1}$ for
which the linear section $V\cap \{\bfs\gamma\cdot x=0\}$ is not
smooth of codimension $s+1$, is contained in a hypersurface of
$(\Pp^n)^{s+1}$ of ``low'' degree.
\begin{theorem}\label{th: definition H}
Let $V\subset \Pp^n$ be a complete intersection defined over $\fq$,
of dimension $r$, multidegree
$\boldsymbol{d}:=(d_1,\ldots,d_{n-r})$, degree $\delta$ and singular
locus of dimension at most $0\le s\le r-2$. Let
$D:=\sum_{i=1}^{n-r}(d_i-1)$. There exists a hypersurface
$\mathcal{H}\subset(\Pp^n)^{s+1}$, defined by a multihomogeneous
polynomial of degree at most $D^{r-s-1}(D+r-s)\delta$ in each group
of variables $\Gamma_i$, with the following property: if
$\boldsymbol{\gamma}\in(\Pp^n)^{s+1}\setminus\mathcal{H}$ and
$\mathcal{L}:=\{\bfs\gamma\cdot x=0\}$, then $V\cap \mathcal{L}$ is
nonsingular of pure dimension $r-s-1$.
\end{theorem}

\begin{proof}
By the version of the Bertini smoothness theorem of, e.g.,
\cite[Proposition 1.3]{GhLa02a}, for generic $\bfs\gamma \in
(\Pp^n)^{s+1}$ and $\mathcal{L}:=\{\bfs\gamma\cdot x=0\}$, the
linear section $V\cap\mathcal{L}$ is nonsingular of pure codimension
$s+1$. Furthermore, we may assume without loss of generality that
the equations
$$F_1=0\klk F_{n-r}=0,\bfs\gamma\cdot X=0,\Delta^1(\bfs\gamma,X)=0
\klk \Delta^{r-s}(\bfs\gamma,X)=0,$$
do not have common solutions in $\Pp^n$, where $\Delta^1\klk
\Delta^{r-s}$ are the polynomials of \eqref{eq: definition W_s'}.
Denote $\K:=\overline{\fq(\Gamma)}$. Then the equations
\begin{equation}\label{eq: definition W_s' bis}
F_1=0\klk F_{n-r}=0,\bfs\Gamma\cdot X=0,\Delta^1(\bfs\Gamma,X)=0
\klk \Delta^{r-s}(\bfs\Gamma,X)=0
\end{equation}
do not have common solutions in the $n$--dimensional projective
space $\Pp_{\K}^n$ over $\K$. As a consequence, the multidimensional
resultant of the corresponding polynomials is a nonzero element of
$\fq[\bfs\Gamma]$ which vanishes on $\bfs\gamma\in(\Pp^n)^{s+1}$ if
and only if the substitution of $\bfs\gamma$ for $\bfs\Gamma$ in
\eqref{eq: definition W_s' bis} yields a nonempty variety of
$\Pp^n$.

Define $d_i:=1$ for $n-r+1\le i\le n-r+s+1$ and $d_i:=D$ for
$n-r+s+2\le i\le n+1$ so that the polynomials in \eqref{eq:
definition W_s' bis} have degree $d_1\klk d_{n+1}$ in $X$
respectively. Set $D_i:=\binom{d_i+n}{n}-1$ for $1\le i\le n+1$ and
denote $\bfs D:=(D_1\klk D_{n+1})$ and $\Pp^{\bfs D}=\Pp^{D_1}
\times \cdots\times\Pp^{D_{n+1}}$. Let $\bfs\Lambda_i$ be a group of
$D_i+1$ indeterminates over $\cfq$ for $1\le i\le n+1$,
$\cfq[\bfs{\Lambda}]:= \cfq[\bfs{\Lambda}_1\klk
\bfs{\Lambda}_{n+1}]$ and let $P\in\cfq[\bfs \Lambda]$ be the
multivariate resultant of generic polynomials of
$\cfq[\bfs\Lambda_1][X]\klk\cfq[\bfs\Lambda_{n+1}][X]$ of degrees
$d_1\klk d_{n+1}$ respectively. Denote by
$\mathcal{H}_{gen}\subset\Pp^{\bfs D}$ the hypersurface defined by
$P$. For $\bfs\gamma\in(\Pp^n)^{s+1}$, the substitution of
$\bfs\gamma$ for $\bfs\Gamma$ in \eqref{eq: definition W_s' bis}
yields a nonempty variety of $\Pp^n$ if and only if the
corresponding $(n+1)$--tuple of polynomials $\big(F_1\klk
F_{n-r},\bfs\gamma\cdot X, \Delta^1(\bfs\gamma,X)\klk
\Delta^{r-s}(\bfs\gamma,X)\big)$ belongs to $\mathcal{H}_{gen}$. Let
$\phi:(\Pp^n)^{s+1}\to\Pp^{\bfs D}$ be the regular mapping defined
as
$$\phi(\bfs\gamma):=\big(F_1\klk F_{n-r},\bfs\gamma\cdot X,
\Delta^1(\bfs\gamma,X)\klk \Delta^{r-s}(\bfs\gamma,X)\big).$$
Finally, let $\mathcal{H}$ be the hypersurface of $(\Pp^n)^{s+1}$
defined by the polynomial $\phi^*(P)$, where
$\phi^*:\cfq[\bfs{\Lambda}]\to\cfq[\bfs{\Gamma}]$ is the
$\cfq$--algebra homomorphism induced by $\phi$. We claim that
$\mathcal{H}$ satisfies the requirements in the statement of the
theorem.

Indeed, let $\bfs\gamma\notin \mathcal{H}$. Then the substitution of
$\bfs\gamma$ for $\bfs\Gamma$ in \eqref{eq: definition W_s' bis}
yields the empty variety of $\Pp^n$. In particular,
$\bfs\gamma\notin \pi_2(\mathcal{W}')$, where $\mathcal{W}'$ is the
variety of Corollary \ref{coro: definition W'}. By the definition of
$\mathcal{W}'$ we have $\bfs\gamma\notin\pi_2(\mathcal{W}'')$, where
$\mathcal{W}''$ is the variety of \eqref{eq: definition W''}. This
implies that $\bfs\gamma\in\mathcal{U}\setminus\pi_2(\mathcal{W})$,
and Lemma \ref{lemma: props section notin pi_2(W)} shows that
$V_{\rm sm}\cap \mathcal{L}$ is smooth of pure codimension $s+1$.
Furthermore, from the definition of $\mathcal{W}''$ it follows that
$\bfs\gamma\notin\pi_2\big((\Sigma\times(\Pp^n)^{s+1})\cap
L_\Gamma\big)$, which implies that $\Sigma\cap
\mathcal{L}=\emptyset$. We conclude that $V\cap\mathcal{L} =V_{\rm
sm}\cap \mathcal{L}$ is smooth of pure codimension $s+1$.

Finally we prove the bound on the multidegree of $\mathcal{H}$ of
the statement of the theorem. According to, e.g., \cite[Chapter 3,
Theorem 3.1]{CoLiOS98}, the multivariate resultant $P\in\cfq[\bfs
\Lambda]$ is a multihomogeneous polynomial with
$$\deg P_{\Lambda_i}=\left\{\begin{array}{cl}
D^{r-s}\delta&\textrm{ for }n-r+1\le i\le n-r+s+1,\\
D^{r-s-1}\delta&\textrm{ for }n-r+s+2\le i\le n+1. \end{array}
\right.$$
The homomorphism $\phi^*:\cfq[\bfs\Lambda]\to\cfq[\bfs\Gamma]$ maps
$\Lambda_{n-r+1+i}$ to $\Gamma_i$ for $0\le i\le s$ and
$\Lambda_{n-r+s+1+i}$ on the vector of coefficients of
$\Delta^i\in\fq[\bfs\Gamma][X]$ for $1\le i\le r-s$. Since each
coefficient of $\Delta^i\in\fq[\bfs\Gamma]$ is homogeneous of degree
$1$ in $\Gamma_j$ for $0\le j\le s$, we see that
$$\deg_{\Gamma_i}\phi^*(P)=\deg_{\Lambda_{n-r+1+i}}P+
\sum_{j=n-r+s+2}^{n+1}\deg_{\Lambda_j}P=D^{r-s}\delta+ (r-s)
D^{r-s-1}\delta.$$
This finishes the proof of the theorem.
\end{proof}

According to Theorem \ref{th: definition H}, for ``most'' elements
$\boldsymbol{\gamma}\in (\Pp^n)^{s+1}$ the linear section
$V\cap\mathcal{L}:=V\cap \{\bfs\gamma\cdot x=0\}$ is nonsingular of
codimension $s+1$. Furthermore, combining Theorem \ref{th:
definition H} with the results of Section \ref{subsec: zeros
multihomogeneous pols} we are able to estimate the number of
``good'' linear sections $V\cap\mathcal{L}$ which are defined over
$\fq$, which is essential for the results of Section \ref{section:
Hooley}. In particular, for $q>D^{r-s-1}(D+r-s)\delta$, Corollary
\ref{coro: existencia no-cero pol multih} proves that there exists
$\boldsymbol{\gamma}\in (\Pp^{n}(\fq))^{s+1}\setminus \mathcal{H}$.
This yields an effective version of the Bertini smoothness theorem,
which may be of independent interest. We remark that this result
will not be used in the sequel.
\begin{theorem}\label{th: Bertini}
Let $V\subset \Pp^n$ be a complete intersection defined over $\fq$,
of dimension $r$, multidegree
$\boldsymbol{d}:=(d_1,\ldots,d_{n-r})$, degree $\delta$ and singular
locus of dimension at most $0\le s\le r-2$. Let
$D:=\sum_{i=1}^{n-r}(d_i-1)$. If $q>D^{r-s-1}(D+r-s)\delta$, then
there exists a linear variety $\mathcal{L}\subset\Pp^n$ defined over
$\fq$ of dimension $n-s-1$ such that the linear section
$V\cap\mathcal{L}$ is nonsingular of pure codimension $s+1$.
\end{theorem}

An effective version of a weak form of a Bertini smoothness theorem
for hypersurfaces is obtained in \cite{Ballico03}. Nevertheless, the
bound given in \cite{Ballico03} is exponentially higher than ours
and therefore not suitable for our purposes, even in the
hypersurface case. On the other hand, in \cite{CaMa07} a version of
the Bertini smoothness theorem for normal complete intersections is
established, which is significantly generalized and improved by
Theorem \ref{th: Bertini}. Finally, the result of Theorem \ref{th:
Bertini} is similar both quantitatively and qualitatively to
\cite[Corollary 6.6]{CaMaPr13}, the main contribution over the
latter being the simplicity of the approach. Nevertheless, neither
Theorem \ref{th: Bertini} nor \cite[Corollary 6.6]{CaMaPr13} provide
enough information on the nonsingular linear sections of $V$ of
codimension $s+1$ defined over $\fq$ for the purposes of Section
\ref{section: Hooley}.
%
%
\section{Estimates on the number of rational points}\label{section: Hooley}
Let $V\subset\Pp^n$ be an ideal--theoretic complete intersection
defined over $\fq$, of dimension $r$, multidegree
$\boldsymbol{d}:=(d_1\klk d_{n-r})$ and singular locus of dimension
at most $0\le s\le r-2$. As before, we denote $\delta:=\deg
V=d_1\cdots d_{n-r}$ and $D:=\sum_{i=1}^{n-r}(d_i-1)$. In this
section we obtain an explicit version of the Hooley--Katz estimate
\eqref{eq: estimate hooley-katz} for $V$.

The proof of \eqref{eq: estimate hooley-katz} in \cite{Hooley91}
proceeds in $s+1$ steps, considering successive hyperplane sections
of V until nonsingular sections are obtained. The number of
$\fq$--rational points of each of these nonsingular sections is
estimated using Deligne's estimate. A key ingredient in
\cite{Hooley91} is an upper bound for the second moment
$$
M_1:=\sum_{\boldsymbol{m}\in\fq^{n+1}}
\big(N-q\,N(\boldsymbol{m})\big)^2,
$$
where $N$ and $N(\boldsymbol{m})$ are the number of $\fq$--rational
points of  $V$ and of the linear section of $V$ determined by the
hyperplane defined by $\boldsymbol{m}$. In this section we introduce
a variant of the second moment $M_1$: the second moment $M_{s+1}$
obtained by considering the linear sections of $V$ determined by all
the linear varieties of codimension $s+1$ of $\Pp^n$ defined over
$\fq$.

First we estimate the number of nonsingular $\fq$--definable linear
sections of $V$ of pure codimension $s+1$.
\begin{lemma} \label{lemma: nonsingular linear sections}
Let $N_{ns}$ be the number of $\boldsymbol{\gamma}\in
(\fq^{n+1})^{s+1}$ for which $V\cap \mathcal{L}$ is nonsingular of
pure codimension $s+1$, where $\mathcal{L}:=\{\bfs\gamma\cdot
x=0\}\subset\Pp^n$. Then
$$N_{ns}\ge(q-d)^{s+1}q^{n(s+1)},$$
where $d:=D^{r-s-1}(D+r-s)\delta$.
\end{lemma}
\begin{proof}
Let $\mathcal{H}\subset(\Pp^n)^{s+1}$ be the hypersurface of the
statement of Theorem \ref{th: definition H}. 
%
The hypersurface $\mathcal{H}$ is defined by a multihomogeneous
polynomial $F\in \cfq[\boldsymbol{\Gamma}]$ of degree at most $d$ in
each group of variables $\Gamma_i$. For any $\boldsymbol{\gamma}\in
(\fq^{n+1})^{s+1}$ with $F(\bfs\gamma)\not=0$, the corresponding
linear section $V\cap\{\bfs\gamma\cdot x=0\}$ is nonsingular of pure
codimension $s+1$. As a consequence, from Proposition \ref{prop:
upper bound multihomogeneous} we obtain
\begin{align*}
N_{ns}&\ge q^{(n+1)(s+1)}
-\sum_{\boldsymbol{\varepsilon}\in\{0,1\}^{s+1}\setminus\{\boldsymbol{0}\}}
(-1)^{|\boldsymbol{\varepsilon}|+1} d^{|\boldsymbol{\varepsilon}|}
q^{(n+1)(s+1)-|\boldsymbol{\varepsilon}|}\\
&=\sum_{\boldsymbol{\varepsilon}\in\{0,1\}^{s+1}}
(-d)^{|\boldsymbol{\varepsilon}|}
q^{(n+1)(s+1)-|\boldsymbol{\varepsilon}|} =
\sum_{i=0}^{s+1}\sum_{\boldsymbol{\varepsilon}:
\,|\boldsymbol{\varepsilon}|=i}(-d)^i q^{(n+1)(s+1)-i}.
\end{align*}
This implies
$$
N_{ns}\ge
q^{n(s+1)}\left(\displaystyle\sum_{i=0}^{s+1}\binom{s+1}{i} (-d)^i
q^{s+1-i}\right) =q^{n(s+1)}(q-d)^{s+1},
$$
which proves the statement of the lemma.
\end{proof}

Now we consider the second moment defined as
\begin{equation}\label{definicion de 2do momento}
M_{s+1}:=\sum_{\boldsymbol{\gamma}\in\fq^{(n+1)(s+1)}}
\big(N-q^{s+1}N(\boldsymbol{\gamma})\big)^2,
\end{equation}
where $N:=|V(\fq)|$, $N(\boldsymbol{\gamma}):=|V\cap \mathcal{L}
(\fq)|$ and $\mathcal{L}:=\{\bfs\gamma \cdot x = 0\}$.
\begin{lemma}\label{lemma: bound second moment}
We have
$M_{s+1}=Nq^{(n+1)(s+1)}(q^{s+1}-1).$
\end{lemma}
\begin{proof}
Set $t:=(n+1)(s+1)$ and observe that
\begin{equation} \label{eq: second moment}
M_{s+1}=\sum_{\boldsymbol{\gamma}\in\fq^t}N^2 -
2q^{s+1}N\sum_{\boldsymbol{\gamma}\in\fq^t}N(\boldsymbol{\gamma})
+q^{2(s+1)}\sum_{\boldsymbol{\gamma}\in\fq^t}N(\boldsymbol{\gamma})^2.
\end{equation}
First we consider the second term in the right--hand side of
(\ref{eq: second moment}):
\begin{equation}\label{eq: second term second moment}
\sum_{\boldsymbol{\gamma}\in\fq^t}N(\boldsymbol{\gamma})=
\sum_{\boldsymbol{\gamma}\in\fq^t} \mathop{\sum_{x\in
V(\fq)}}_{\boldsymbol{\gamma}\cdot x=0}1 = \sum_{x\in
V(\fq)}\mathop{\sum_{\boldsymbol{\gamma}\in\fq^t}}_{\boldsymbol{\gamma}\cdot
x=0} 1=q^{t-s-1}\,N.
\end{equation}

On the other hand, concerning the third term of the right--hand side
of (\ref{eq: second moment}),
$$\sum_{\boldsymbol{\gamma}\in\fq^t}N(\boldsymbol{\gamma})^2
=\sum_{\boldsymbol{\gamma}\in\fq^t} \Bigg(\mathop{\sum_{x\in
V(\fq)}}_{\boldsymbol{\gamma}\cdot x=0}1\Bigg)
\Bigg(\mathop{\sum_{x'\in V(\fq)}}_{\boldsymbol{\gamma}\cdot
x'=0}1\Bigg)= \sum_{\boldsymbol{\gamma}\in\fq^t}\Bigg(
\mathop{\sum_{x\in V(\fq)}}_{\boldsymbol{\gamma} \cdot x=0}1+
\mathop{\sum_{x, x'\in V(\fq),\,x\neq x'}}_{\boldsymbol{\gamma}\cdot
x=\,\boldsymbol{\gamma}\cdot x'=0}1\Bigg).
$$
Further, we have
\begin{align*}\sum_{\boldsymbol{\gamma}\in\fq^t}\sum_{\substack{x, x'\in
V(\fq),\,x\neq x'\\\boldsymbol{\gamma}\cdot
x=\,\boldsymbol{\gamma}\cdot x'=0}}1=\mathop{\sum_{x, x'\in
V(\fq)}}_{ x\neq x'}
\sum_{\substack{\boldsymbol{\gamma}\in\fq^t\\\boldsymbol{\gamma}\cdot
x=\,\boldsymbol{\gamma}\cdot x'=0}}1&= \mathop{\sum_{x, x'\in
V(\fq)}}_{ x\neq x'}q^{t-2(s+1)}\\&=q^{t-2(s+1)}N(N-1). \end{align*}
We conclude that
\begin{equation}\label{eq: third term of second moment}
\sum_{\boldsymbol{\gamma}\in\fq^t}N(\boldsymbol{\gamma})^2=q^{t-s-1}
N+q^{t-2(s+1)}N(N-1).
\end{equation}

Combining (\ref{eq: second moment}), (\ref{eq: second term second
moment}) and (\ref{eq: third term of second moment}) 
%
we easily deduce the statement of the lemma.
\end{proof}

From Lemma \ref{lemma: bound second moment} we deduce that there are
at least $\frac{1}{2}q^{(n+1)(s+1)}$ elements
$\boldsymbol{\gamma}\in\fq^{(n+1)(s+1)}$ such that the linear
variety $\mathcal{L}:=\{\bfs\gamma\cdot x=0\}$ satisfies the
condition
$$\big||V(\fq)|-q^{s+1}|(V\cap \mathcal{L}) (\fq)|\big|
\le \sqrt{2 N (q^{s+1}-1)}.$$
Otherwise,
$\big||V(\fq)|-q^{s+1}|(V\cap\mathcal{L}) (\fq)|\big|> \sqrt{2 N
(q^{s+1}-1)}$
for at least $\frac{1}{2}q^{(n+1)(s+1)}$ linear varieties
$\mathcal{L}$ defined over $\fq$, and then
$$M_{s+1}>N(q^{s+1}-1)\,q^{(n+1)(s+1)},$$
which contradicts Lemma \ref{lemma: bound second moment}. In other
words, we have the following result.
\begin{corollary}\label{coro: property of the sections}
There exist at least $\frac{1}{2}q^{(n+1)(s+1)}$ elements
$\boldsymbol{\gamma}\in \fq^{(n+1)(s+1)}$ such that the linear
variety $\mathcal{L}:=\{\bfs\gamma\cdot x=0\}$ satisfies the
condition
\begin{equation}\label{eq: Hooley 's condition}
\big||V(\fq)|-q^{s+1}|(V\cap\mathcal{L}) (\fq)|\big| \leq \sqrt{2 N
(q^{s+1}-1)}.
\end{equation}
\end{corollary}

Denote $d:=D^{r-s-1}(D+r-s)\delta$. According to Lemma \ref{lemma:
nonsingular linear sections}, there exist at least
$(q-d)^{s+1}q^{n(s+1)}$ elements
$\boldsymbol{\gamma}\in(\fq^{n+1})^{s+1}$ such that the linear
section of $V\cap\mathcal{L}$ defined by
$\mathcal{L}:=\{\boldsymbol{\gamma}\cdot x=0\}$ is nonsingular of
codimension $s+1$. In particular, for
\begin{equation} \label{ineq: deduccion de la condicion de Hooley}
(q-d)^{s+1}q^{n(s+1)}> \frac{1}{2}q^{(n+1)(s+1)},
\end{equation}
there exists a nonsingular $\fq$--definable linear section
$V\cap\mathcal{L}$ of codimension $s+1$ satisfying (\ref{eq: Hooley
's condition}).

Observe that (\ref{ineq: deduccion de la condicion de Hooley}) is
equivalent to the inequality $(1-\tfrac{d}{q})^{s+1}>\tfrac{1}{2}$.
By the Bernoulli inequality, $(1-\tfrac{d}{q})^{s+1}\geq
1-(s+1)\tfrac{d}{q}$ for $q>d$. Therefore, the condition
$1-(s+1)\tfrac{d}{q}>\tfrac{1}{2}$ implies (\ref{ineq: deduccion de
la condicion de Hooley}). As a consequence, we obtain the following
result.
\begin{corollary}\label{coro: existence of the good linear section}
For $q> 2(s+1)D^{r-s-1}(D+r-s)\delta$, there exists a nonsingular
$\fq$--definable linear section of $V$ of codimension $s+1$ which
satisfies (\ref{eq: Hooley 's condition}).
\end{corollary}

Finally, we are ready to state our estimate on the number of
$\fq$--rational points of a singular complete intersection.
\begin{theorem}\label{th: estimate hooley}
Let $q> 2(s+1)D^{r-s-1}(D+r-s)\delta$ and  $V\subset\Pp^n$ be a
complete intersection defined over $\fq$, of dimension $r$, degree
$\delta$, multidegree $\boldsymbol{d}:=(d_1\klk d_{n-r})$ and
singular locus of dimension at most $0\le s\le r-2$. Then
\begin{equation}\label{eq: estimate hooley}
\big||V(\fq)|-p_r\big|\leq \big(b'_{r-s-1}+2
\sqrt{\delta}+1\big)\,q^{\frac{r+s+1}{2}},
\end{equation}
where $b'_{r-s-1}:=b'_{r-s-1}(n-s-1,\boldsymbol{d})$ is the
$(r-s-1)$th primitive Betti number of any nonsingular complete
intersection of $\Pp^{n-s-1}$ of dimension $r$ and multidegree
$\boldsymbol{d}$.
\end{theorem}
\begin{proof}
Since $q> 2(s+1)D^{r-s-1}(D+r-s)\delta$, by Corollary \ref{coro:
existence of the good linear section} there exists
$\boldsymbol{\gamma}\in\fq^{(n+1)(s+1)}$ such that the linear
section $V\cap\mathcal{L}$ defined by
$\mathcal{L}:=\{\bfs\gamma\cdot x=0\}$ is nonsingular of dimension
$r-s-1$ and satisfies
$$\big||V(\fq)|-q^{s+1}|(V\cap \mathcal{L}) (\fq)|\big|   \leq
\sqrt{2 N (q^{s+1}-1)}.$$
Fix such an element $\boldsymbol{\gamma}\in\fq^{(n+1)(s+1)}$. We
have
$$\big||V(\fq)|-p_r\big|\leq\big||V(\fq)|-q^{s+1}
|V\cap \mathcal{L} (\fq)|\big|+ \big|q^{s+1}|V\cap \mathcal{L}
(\fq)|-p_r\big|.$$
By the definition of $\bfs\gamma$ and the identity
$p_r=q^{s+1}p_{r-s-1}+p_s$, it follows that
$$\big||V(\fq)|-p_r\big|\leq \sqrt{2 N (q^{s+1}-1)}
+ q^{s+1}\big||V\cap \mathcal{L} (\fq)|-p_{r-s-1}\big|+ p_s.$$
Since $V\cap\mathcal{L}$ is a nonsingular complete intersection of
$\mathcal{L}$ of dimension $r-s-1$ and multidegree $\bfs d$,
applying \eqref{eq: estimate deligne intro} we obtain
$$\big||V(\fq)|-p_r\big|\leq \sqrt{2 N (q^{s+1}-1)}
+ b'_{r-s-1}(n-s-1,\boldsymbol{d})\,q^{\frac{r+s+1}{2}}+p_s.$$
By the bound $N\le \delta p_r$ and elementary calculations, the
theorem follows.
\end{proof}

Let $V\subset\Pp^n$ be a singular complete intersection as in the
statement of Theorem \ref{th: estimate hooley}. In \cite[Theorem
6.1]{GhLa02a}, the following estimate is obtained:
\begin{equation}\label{eq: estimate GL}
\big||V(\fq)|-p_r\big|\le b_{r-s-1}'\,q^{\frac{r+s+1}{2}}+ 9\cdot
2^{n-r}\big((n-r)d+3\big)^{n+1}q^{\frac{r+s}{2}},
\end{equation}
where $d:=\max_{1\le i\le n-r}d_i$. Regarding \eqref{eq: estimate
hooley} and \eqref{eq: estimate GL} one observes that the error term
in \eqref{eq: estimate hooley} avoids the exponential dependency on
$n$ present in \eqref{eq: estimate GL}. On the other hand,
\eqref{eq: estimate GL} holds without any condition on $q$, while
\eqref{eq: estimate hooley} is valid for
$q>2(s+1)D^{r-s-1}(D+r-s)\delta$.
%
%
%
\subsection{Normal complete intersections}
Let $V\subset\Pp^n$ be a complete intersection defined by $\fq$, of
dimension $r$, degree $\delta$, multidegree $\bfs{d}$ and the
singular locus of codimension at least $2$. By the case $s=r-2$ of
Theorem \ref{th: estimate hooley} we conclude that, if
$q>2(r-1)D(D+2)\delta$, then
$$
\big||V(\fq)|-p_r\big|\leq\big(b_1'(n-r+1,\bfs{d})
+2\sqrt{\delta}+1\big)\,q^{r-\frac{1}{2}}.
$$
Nevertheless, the condition on $q$ may restrict the range of
applicability of this estimate. For this reason, the next result
provides a further estimate which holds without restrictions on $q$.
\begin{corollary}\label{coro: estimate normal variety}
Let $V\subset\Pp^n$ be a normal complete intersection defined over
$\fq$, of dimension $r\geq 2$, degree $\delta$ and multidegree
$\bfs{d}$. Then
\begin{equation}\label{eq: estimate normal variety}
\big||V(\fq)|-p_r\big|\le
3\,r^{1/2}(D+1)\delta^{3/2}\,q^{r-\frac{1}{2}}.
\end{equation}
\end{corollary}
\begin{proof}
Suppose first that $q>2(r-1)D(D+2)\delta$. Since
$b'_{1}(n-r+1,\bfs{d})=(D-2)\delta+2$ (see, e.g., \cite[Theorem
4.1]{GhLa02a}), Theorem \ref{th: estimate hooley} readily implies
the corollary. As a consequence, we may assume $q\le
2(r-1)D(D+2)\delta$. By \eqref{eq: upper bound -- projective gral},
it follows that $|V(\fq)| \le \delta p_r$. Therefore,
$$
\big||V(\fq)|-p_r\big|\leq (\delta -1)p_r \leq 2\delta q^r \le
3\,r^{1/2}(D+1)\delta^{3/2}q^{r-1/2}.
$$
This finishes the proof of the corollary.
\end{proof}

Let $V\subset\Pp^n$ be a normal complete intersection as in
Corollary \ref{coro: estimate normal variety}. According to
\cite[Corollary 6.2]{GhLa02a},
\begin{equation}\label{eq: estimate GL normal var for estimate}
\big||V(\fq)|-p_r\big|\le (\delta(D-2)+2)q^{r-1/2}+9\cdot
2^{n-r}((n-r)d+3)^{n+1}q^{r-1},
\end{equation}
where $d:=\max_{1\le i\le n-r}d_i$. On the other hand,
\cite[Corollary 8.3]{CaMaPr13} shows that
\begin{equation}\label{eq: estimate CMP normal var}
\big||V(\fq)|-p_r\big|\le(\delta(D-2)+2)q^{r-1/2}
+14D^2\delta^2q^{r-1}.
\end{equation}
These are the most accurate estimates to the best of our knowledge.

For the sake of comparison, it can be seen that
$$
2^{n-r}((n-r)d+3)^{n+1}\ge \big(2(n-r)\big)^{n-r}D^{r+1}\delta.
$$
This shows that for varieties of high dimension, say $r\ge (n+1)/2$,
(\ref{eq: estimate normal variety}) and (\ref{eq: estimate CMP
normal var}) are clearly preferable to (\ref{eq: estimate GL normal
var for estimate}). In particular, for hypersurfaces the error term
in both (\ref{eq: estimate normal variety}) and (\ref{eq: estimate
CMP normal var}) is at most quartic in $\delta$, while that of
(\ref{eq: estimate GL normal var for estimate}) contains an
exponential term $\delta^{n+1}$. On the other hand, for varieties of
low dimension (\ref{eq: estimate GL normal var for estimate}) might
be more accurate than both (\ref{eq: estimate normal variety}) and
(\ref{eq: estimate CMP normal var}). In this sense, we may say that
(\ref{eq: estimate normal variety})--(\ref{eq: estimate CMP normal
var}) somewhat complement (\ref{eq: estimate GL normal var for
estimate}). Finally, the right--hand side of (\ref{eq: estimate
normal variety}) depends on a lower power of $\delta$ than that of
(\ref{eq: estimate CMP normal var}), which may yield a significant
improvement in estimates for varieties of large degree.


\end{document}

%